\documentclass[12pt]{article}

\usepackage{mathtools}
\usepackage[demo]{graphicx} 
\usepackage{hyperref}
\usepackage{multirow}
\usepackage{float}
\usepackage{diagbox}
\usepackage{indentfirst}
\usepackage{tikz}
\usepackage{dsfont}
\usepackage{amssymb}
\usepackage{amsmath,mathrsfs,amsfonts}
\usepackage{graphics}
\usepackage{amsthm}
\usepackage{setspace}
\usepackage{setspace}
\usepackage[mathscr]{eucal}
\usepackage{caption}
\usepackage{subcaption}
\usepackage{rotating}
\usetikzlibrary{decorations.pathreplacing}
\newtheorem{theo}{Theorem}

\newtheorem{lem}[theo]{Lemma}
\newtheorem{corollary}[theo]{Corollary}

\newtheorem{problem}[theo]{Problem}
 \theoremstyle{definition}

\theoremstyle{remark}
\newtheorem{rema}[theo]{Remark}

\def\p{\overrightarrow{p}}

\newcounter{casenum}[theo]

\newcounter{subcasenum}[theo]

\newcounter{claimnum}[theo]

\allowdisplaybreaks

\begin{document}
\thispagestyle{plain}

\begin{center}
{\large Connected graphs minimizing the spectral radius for\\ given order and dissociation number}
\end{center}
\pagestyle{plain}
\begin{center}
	{
		{\small  Zejun Huang, Jiahui Liu, Chenxi Yang \footnote{  Corresponding author. \\ Email: zejunhuang@szu.edu.cn (Huang), mathjiahui@163.com (Liu), yangchenxi2022@email.szu.edu.cn (Yang)}}\\[3mm]
		{\small   School of Mathematical Sciences, Shenzhen University, Shenzhen 518060, China }\\
		
	}
\end{center}
\begin{center}
\begin{minipage}{140mm}
\begin{center}
{\bf Abstract}
\end{center}
A dissociation set in a graph is a subset of vertices which induces a subgraph with maximum  degree at most one. The dissociation number of a graph  is the maximum cardinality of its dissociation sets. In this paper, we consider the  $n$-vertex connected graphs with a given dissociation number that attain the minimum spectral radius. By using structure analysis and constructing difference equations, we characterize the extremal graphs with dissociation number $n-3$.\\
{\small
{\bf Keywords:} connected graph; dissociation number; extremal graph; spectral radius
}

\end{minipage}
\end{center}
\section{Introduction}

Graphs in this paper are simple, connected and undirected. Let $G = (V(G),E(G))$ be a graph. We denoted  by $d_G(v)$ or $d(v)$  the degree of a vertex $v\in V(G)$,   $G[S]$ the subgraph of $G$ induced by a subset $S\subseteq V(G)$, $G-S$  the induced subgraph $G[V\setminus S]$, which is also written as  $G-v$ when $S=\{v\}$ is a singleton. Let $\rho(G)$ be the spectral radius of $G$, which is the spectral radius of the adjacency matrix $A_G$ of $G$. For a symmetric matrix $B$, we denote by $\lambda_1(B)$ the largest eigenvalue of $B$. By the famous Perron-Frobenius theorem, we have $\rho(G)=\lambda_1(A_G)$ and it has a unique unit eigenvector with all components positive, which is called the {\it Perron vector } of $G$.

Given a set of graphs, it is natural to consider the possible values of certain parameters on these graphs. In 1986, Brualdi and Solheid \cite{BS} posed the problem on determining the maximum or minimum eigenvalue among all matrices in a subset of 0-1 matrices of order $n$. Notably, finding the maximum spectral radius of adjacency matrices for a certain class of graphs and characterizing the extremal graphs is a specific subproblem of this broader question. This problem has attracted significant research attention. The extremal graphs that attain the maximal spectral radius for specific invariants such as independence numbers, clique numbers, matching numbers, diameters  and domination numbers  have been extensively studied in previous works \cite{Feng,FENG2007133,HS,JL,SAH,vD}.

 For the minimization part of the Brualdi-Solheid type problem, it is much more challenging and complicated.   In 2009, Xu, Hong, Shu and Zhai \cite{XU2009937}  studied this type of problem for  connected graphs of order $n$ with a given independence number $\alpha$. They solved the problem for  $\alpha\in\{1,~2,~\lceil n/2\rceil,~\lceil n/2\rceil+1,~n-3,~n-2,~n-1\}$; Du and Shi \cite{Du2013GraphsWS} solved the problem when  $\alpha=3,4$ and the order $n$ is divided by $\alpha$;   Lou and Guo \cite{LOU2022112778} solved the problem for the case  $\alpha=n-4$ and they proved that the extremal graphs are trees when $\alpha\ge \lceil n/2\rceil$;   Hu, Huang and Lou \cite{hu2022graphs} established the structural properties and provided a constructive theorem to determine the  connected graphs of order $n$ attaining the minimum spectral radius when $\alpha\ge\left\lceil n/2\right\rceil$, and they explicitly identified these graphs along with their spectral radius for the cases $\alpha=n-5,n-6$; Choi and Park \cite{choi2023minimal}  solved the problem for the case $\alpha=\lceil n/2\rceil-1$. For other cases, the problem is still open.

A subset $S\subseteq V(G)$ is a {\it dissociation set} if it induces a subgraph with maximum degree at most 1. A {\it maximum dissociation set} is a dissociation set with maximum cardinality, whose cardinality is its {\it dissociation number}, denoted $\operatorname{diss}(G)$. Problems on dissociation number and dissociation sets have been extensively investigated; see \cite{bock2022bound,bock2022relating,BPD,das2023spectral,ORLOVICH20111352,SL,TU2022127107, TZS, doi:10.1137/0210022}. Notice that the dissociation number is a generalization of the independence number. It is  natural   to study the following problem of characterizing the extremal connected graphs of order $n$ with a given dissociation number.
\begin{problem}\label{pro1}
Let $n,\psi$ be integers such that $2\le \psi \le n$.
Denote by $ \mathcal{G}_{n,\psi}$ the connected graphs of order $n$ with dissociation number $\psi$. Characterize the graphs in  $ \mathcal{G}_{n,\psi}$ that attain the minimum spectral radius.
\end{problem}

Huang, Li and Zhan \cite {HLi} solved this problem for the cases $\psi\in \{ 2,\left\lceil 2n/3\right\rceil,n-2,n-1\}$; Huang, Liu and Zhang \cite{HLZ} independently settled the same cases by using a different approach; Zhao, Liu and Xiong \cite{ZLX} solved the problem for the case $\psi=\left\lceil 2n/3\right\rceil-1$. In this paper, we solve Problem \ref{pro1} for the case $\psi=n-3$.

\par Denoted by $G(a, b, c; p, q, r)$ the graph obtained from the path $P_7=v_1v_2v_3v_4v_5v_6v_7$ by attaching $a$ leaves and $p$ edges, $b$ leaves and $q$ edges, $c$ leaves and $r$ edges to the vertices $v_1,v_4,v_7$, respectively; see Figure \ref{fig1}. Our main result  states as follows.
\begin{theo}{\label{theo1}}
Let $n\ge 39$ be an integer and $m=\lfloor {n}/{6}\rfloor$. If $G\in\mathcal{G}_{n,n-3}$ attains the minimum spectral radius, then
\[
G\cong
\begin{cases}
G_{m,0}\equiv G(1,0,0;m-1,m-2,m-1), &\text{if } n\equiv 0 ~(\it{mod}\ 6);\\
G_{m,1}\equiv G(1,0,1;m-1,m-2,m-1), &\text{if } n\equiv 1 ~(\it{mod}\ 6);\\
G_{m,2}\equiv G(1,0,0;m-1,m-2,m), &\text{if } n\equiv 2 ~(\it{mod}\ 6);\\
G_{m,3}\equiv G(0,0,0;m,m-2,m), &\text{if } n\equiv 3 ~(\it{mod}\ 6);\\
G_{m,4}\equiv G(0,1,0;m,m-2,m), &\text{if } n\equiv 4 ~(\it{mod}\ 6);\\
G_{m,5}\equiv G(0,0,0;m,m-1,m), &\text{if } n\equiv 5 ~(\it{mod}\ 6).
\end{cases}
\]
\end{theo}

We will prepare some lemmas in Section 2 and present the proof of Theorem \ref{theo1} in Section 3.
 Our strategy combines structure analysis with traditional techniques in spectral graph theory. Unlike prior approaches that apply structural analysis directly to the original graphs, we focus on some hypergraphs derived from dissociation sets. A further key distinction   is our adoption of a method for constructing difference equations related to the Perron vector of a graph, an approach first proposed in
\cite{SK}.
\begin{figure}[H]
			\centering
			\begin{tikzpicture}
				\draw(0,0)[fill]circle[radius=1.0mm]--(-0.5,1)[fill]circle[radius=1.0mm];
				\draw(1,0)[fill]circle[radius=1.0mm];
				\draw(2,0)[fill]circle[radius=1.0mm];
				\draw(3,0)[fill]circle[radius=1.0mm]--(2.5,1)[fill]circle[radius=1.0mm];
				\draw(3,0)[fill]circle[radius=1.0mm]--(3.5,1)[fill]circle[radius=1.0mm];
				\draw(4,0)[fill]circle[radius=1.0mm];
				\draw(5,0)[fill]circle[radius=1.0mm];
				\draw(6,0)[fill]circle[radius=1.0mm]--(6.5,1)[fill]circle[radius=1.0mm];
				\draw(6,0)[fill]circle[radius=1.0mm]--(5.5,1)[fill]circle[radius=1.0mm];
				\draw(0,0)--(1,0);
				\draw (1,0)--(2,0);
				\draw (2,0)--(3,0);
				\draw (3,0)--(4,0);
				\draw (4,0)--(5,0);
				\draw (5,0)--(6,0);
				\draw(-0.4,-0.8)[fill]circle[radius=1.0mm];\draw(0.4,-0.8)[fill]circle[radius=1.0mm];\draw(-0.4,-1.6)[fill]circle[radius=1.0mm];\draw(0.4,-1.6)[fill]circle[radius=1.0mm];
				\draw(2.6,-0.8)[fill]circle[radius=1.0mm];\draw(3.4,-0.8)[fill]circle[radius=1.0mm];\draw(2.6,-1.6)[fill]circle[radius=1.0mm];\draw(3.4,-1.6)[fill]circle[radius=1.0mm];
				\draw(5.6,-0.8)[fill]circle[radius=1.0mm];\draw(6.4,-0.8)[fill]circle[radius=1.0mm];\draw(5.6,-1.6)[fill]circle[radius=1.0mm];\draw(6.4,-1.6)[fill]circle[radius=1.0mm];
				\draw(0,0)--(-0.4,-0.8);\draw (0,0)--(0.4,-0.8);\draw (3.0,0)--(2.6,-0.8);\draw (3,0)--(3.4,-0.8);\draw (6.0,0)--(5.6,-0.8);\draw (6,0)--(6.4,-0.8);
				\draw(-0.4,-0.8)--(-0.4,-1.6);\draw (0.4,-0.8)--(0.4,-1.6);\draw (2.6,-0.8)--(2.6,-1.6);\draw (3.4,-0.8)--(3.4,-1.6);\draw (5.6,-0.8)--(5.6,-1.6);\draw (6.4,-0.8)--(6.4,-1.6);
				\node at(0,-1.6) {$\dots$};\node at(3,-1.6) {$\dots$};\node at(6,-1.6) {$\dots$};\node[below] at(0,-1.6) {$p$};\node[below] at(3,-1.6) {$q$};\node[below] at(6,-1.6) {$r$};
				\draw(0.5,1)[fill]circle[radius=1.0mm];
				\draw(0,0)--(0.5,1);
				\node at(0,1)[above=2pt] {$a$};
				\node at(0,1){$\cdots$};
				\node at(3,1)[above=2pt] {$b$};
				\node at(3,1){$\cdots$};
				\node at(0,0)[left=2pt] {$v_1$};
				\node at(1,0)[below=2pt] {$v_2$};
				\node at(2,0)[below=2pt] {$v_3$};
				\node at(3,0)[below=4pt] {$v_4$};
				\node at(4,0)[below=2pt] {$v_5$};
				\node at(5,0)[below=2pt] {$v_6$};
				\node at(6,0)[right=2pt] {$v_7$};
				\node at(6,1)[above=2pt] {$c$};
				\node at(6,1){$\cdots$};
			\end{tikzpicture}
	\caption {$G(a,b,c;p,q,r)$}
	\label{fig1}
\end{figure}
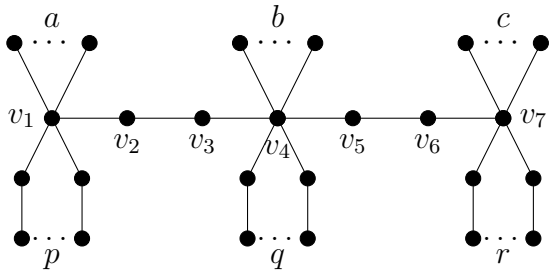
\section{Priliminaries}
In this section, we present some notations and preliminary lemmas.
\begin{lem}\cite{Zhan}\label{PF} If $A$ is an irreducible nonnegative matrix of order $n$ with $n\geq 2$, then the follow statements hold.\\
   \indent (1) $\rho(A)>0$, and $\rho(A)$ is a simple eigenvalue of $A$.\\
   \indent (2) $A$ has a positive eigenvector corresponding to $\rho(A)$.\\
   \indent (3) All nonnegative eigenvectors of $A$ correspond to the eigenvalue $\rho(A)$.\\
\end{lem}

\begin{lem}\cite{Zhan}\label{W_ieq} Let $A,B$ be two $n\times n$ symmetric nonnegative matrices. Then
    $$max\{\rho(A),\rho(B)\}\leq \rho(A+B)\leq \rho(A)+\rho(B).$$
\end{lem}
\begin{lem}\cite{Cvetkovi1995}\label{lemma13}
Let $S_t$ be the star graph of order $t+1$. Then $\rho(S_t)=\sqrt{t}$.
\end{lem}
\begin{lem}(\cite{CR})\label{lemm13}
	Let $G_1,G_2$ be two graphs, If $P(G_2,\lambda)\geq P(G_1,\lambda)$
	for all $\lambda \geq \rho(G_1)$, then $\rho(G_2)\leq \rho(G_1)$, with equality if and only if $P(G_2,\rho(G_1))= P(G_1,\rho(G_1))=0$.
\end{lem}
\begin{lem}\cite{HLC}\label{lemma14}
	Let $A_1,A_2$ be  real matrices such that $\rho(A_i)$ equals   the maximum real eigenvalue of $A_i$, and $\rho(A_i)\leq \rho,~i=1,2$. If $P(A_2,\lambda)\geq P(A_1,\lambda)$
	for all $\lambda \in [\rho(A_1),\rho]$, then $\rho(A_2)\leq \rho(A_1)$, with equality if and only if $P(A_2,\rho(A_1))=P(A_1,\rho(A_1))=0$.
\end{lem}
\begin{lem}
	\label{lemma3}\cite{Hoffman_1975} If $G_2$ is a proper subgraph of a graph $G_1$, then $\rho(G_1)>\rho(G_2)$.		
\end{lem}
\begin{lem}
	\label{lemma2}\cite{0} Let $v$ be a vertex in a connected graph  $G$, and let $k, m$ be nonnegative integers such that $k\ge m$. Denote by $G^{k,m}$ the graph obtained from $G$ by attaching   two new paths   $vv_1v_2\cdots v_k$ and  $vu_1u_2\cdots u_m$  to $v$, where $v_1,\ldots, v_k,u_1,\ldots,u_m$ are distinct and $\{v_1,\ldots,v_k,u_1,\ldots,u_m\}\cap V(G)=\emptyset$.  Then  $\rho(G^{k,m})>\rho(G^{k+1,m-1})$.
\end{lem}
\par An {\it internal path} of a graph $G$ is a sequence of vertices $v_1, v_2,\ldots, v_l$ with $l\ge2$ such that:\\
\indent (1) the vertices in the sequence are distinct(except possibly $v_1=v_l$);\\
\indent (2) $v_i$ is adjacent to $v_{i+1}(i=1,2,\ldots,l-1)$;\\
\indent (3) the vertex degrees satisfy $d(v_1)\ge3$, $d(v_2)=\cdots=d(v_{l-1})=2$ (unless $l=2$) and \indent $d(v_l)\ge3$.

Let $uv$ be an edge of a graph $G$. The {\it subdivision} of the edge $uv$ is replacing $uv$ with a 2-path, i.e., deleting $uv$ and adding a new vertex $w$ and two new edges $uw$, $wv$.  Denoted by  $G^{uv}$ the graph obtained from $G$ by doing a subdivision of $uv$.
\begin{lem}\label{lemma11}
	\cite{0} Suppose that $G\ncong \tilde{W}_{n}$ and $uv$ is an edge on an internal path of $G$.  Then $\rho(G^{uv})<\rho(G)$.
\end{lem}
\begin{lem}
	\label{lemma6}\cite{Smith1970} The only connected graphs on $n$ vertices with spectral radius less than 2 are the path $P_n$ and
	$W_n$, with additional cases $E_6, E_7, E_8$ when $n=6,7,8$, where $W_n, E_6, E_7, E_8$ have the following diagrams.
\end{lem}
\begin{figure}[H]
	\begin{subfigure}{0.4\textwidth}
		\centering
		\begin{tikzpicture}
			\draw(0,0)[fill]circle[radius=1.0mm]--(1,0)[fill]circle[radius=1.0mm];
			\draw(1,0)[fill]circle[radius=1.0mm]--(2,0)[fill]circle[radius=1.0mm];
			\draw(1,0)[fill]circle[radius=1.0mm]--(1,1)[fill]circle[radius=1.0mm];
			\node at (2.5,0){$\cdots$};
			\draw(3,0)[fill]circle[radius=1.0mm];
			\draw(3,0)--(4,0)[fill]circle[radius=1.0mm];
		\end{tikzpicture}
		\caption*{$W_n$ }
		\label{fig:1}
	\end{subfigure}
	\begin{subfigure}{0.6\textwidth}
		\centering
		\begin{tikzpicture}
			\draw(0,0)[fill]circle[radius=1.0mm]--(0.8,0)[fill]circle[radius=1.0mm];
			\draw(0.8,0)[fill]circle[radius=1.0mm]--(1.6,0)[fill]circle[radius=1.0mm];	
			\draw(1.6,0)[fill]circle[radius=1.0mm]--(2.4,0)[fill]circle[radius=1.0mm];
			\draw(2.4,0)[fill]circle[radius=1.0mm]--(3.2,0)[fill]circle[radius=1.0mm];
			\draw(1.6,0.8)[fill]circle[radius=1.0mm]--(1.6,0)[fill]circle[radius=1.0mm];
		\end{tikzpicture}
		\caption*{$E_6$ }
		\label{fig:sub2.1}
	\end{subfigure}
	\begin{subfigure}{0.4\textwidth}
		\centering
		\begin{tikzpicture}
			\draw(0,0)[fill]circle[radius=1.0mm]--(0.8,0)[fill]circle[radius=1.0mm];
			\draw(0.8,0)[fill]circle[radius=1.0mm]--(1.6,0)[fill]circle[radius=1.0mm];	
			\draw(1.6,0)[fill]circle[radius=1.0mm]--(2.4,0)[fill]circle[radius=1.0mm];
			\draw(2.4,0)[fill]circle[radius=1.0mm]--(3.2,0)[fill]circle[radius=1.0mm];
			\draw(3.2,0)[fill]circle[radius=1.0mm]--(4,0)[fill]circle[radius=1.0mm];
			\draw(1.6,0.8)[fill]circle[radius=1.0mm]--(1.6,0)[fill]circle[radius=1.0mm];
		\end{tikzpicture}
		\caption*{$E_7$}
		\label{fig:sub2.2}
	\end{subfigure}
	\begin{subfigure}{0.6\textwidth}
		\centering
		\begin{tikzpicture}
			\draw(0,0)[fill]circle[radius=1.0mm]--(0.8,0)[fill]circle[radius=1.0mm];
			\draw(0.8,0)[fill]circle[radius=1.0mm]--(1.6,0)[fill]circle[radius=1.0mm];	
			\draw(1.6,0)[fill]circle[radius=1.0mm]--(2.4,0)[fill]circle[radius=1.0mm];
			\draw(2.4,0)[fill]circle[radius=1.0mm]--(3.2,0)[fill]circle[radius=1.0mm];
			\draw(3.2,0)[fill]circle[radius=1.0mm]--(4,0)[fill]circle[radius=1.0mm];
			\draw(4,0)[fill]circle[radius=1.0mm]--(4.8,0)[fill]circle[radius=1.0mm];
			\draw(1.6,0.8)[fill]circle[radius=1.0mm]--(1.6,0)[fill]circle[radius=1.0mm];
		\end{tikzpicture}
		\caption*{$E_8$ }
		\label{fig:sub2.3}
	\end{subfigure}
	\caption{The graphs $W_n$, $E_6,~E_7,~E_8$ }
	\label{fig:2}
\end{figure}
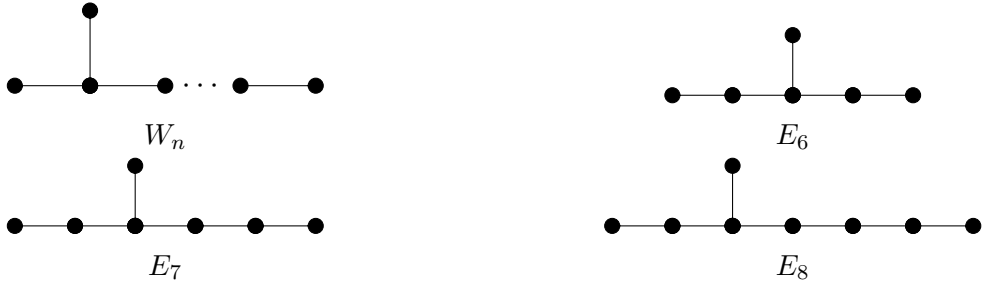

\begin{lem}
	\label{lemma7}\cite{Smith1970} The only connected graphs on $n$ vertices with spectral radius 2 are the cycle   $C_n$ and  $\tilde{W}_{n}$, with additional cases $\tilde{E}_6, \tilde{E}_7, \tilde{E}_8$ when $n=6,7,8$, where $\tilde{W}_{n}, \tilde{E}_6, \tilde{E}_7$ and $ \tilde{E}_8$ have the following diagrams.
\end{lem}
\begin{figure}[H]
	\centering
	\begin{subfigure}{0.45\textwidth}
		\vspace{15pt}
		\centering
		\begin{tikzpicture}
			\draw(0,0)[fill]circle[radius=1.0mm]--(1,0)[fill]circle[radius=1.0mm];
			\draw(1,0)[fill]circle[radius=1.0mm]--(2,0)[fill]circle[radius=1.0mm];
			\draw(3,0)[fill]circle[radius=1.0mm]--(4,0)[fill]circle[radius=1.0mm];
			\draw(4,0)[fill]circle[radius=1.0mm]--(5,0)[fill]circle[radius=1.0mm];
			\draw(1,1)[fill]circle[radius=1.0mm]--(1,0)[fill]circle[radius=1.0mm];
			\draw(4,1)[fill]circle[radius=1.0mm]--(4,0)[fill]circle[radius=1.0mm];
			\node at (2.5,0){$\cdots$};
		\end{tikzpicture}
		\caption*{$\tilde{W}_{n}$}
		\label{fig:3}
	\end{subfigure}
	\begin{subfigure}{0.45\textwidth}
		\centering
		\begin{tikzpicture}
			\foreach \y in {0}{
				\foreach \x in {0,0.8,1.6,2.4}{
					\draw (\x, \y)[fill]circle[radius=1.0mm]--(\x+0.8,\y);
					
				}
			};
			\draw(3.2,0)[fill]circle[radius=1.0mm];
			\draw(1.6,0.8)[fill]circle[radius=1.0mm]--(1.6,0);
			\draw(1.6,1.6)[fill]circle[radius=1.0mm]--(1.6,0.8);
		\end{tikzpicture}
		\caption*{$\tilde{E}_6$}
		\label{fig:4.1}
	\end{subfigure}
	\begin{subfigure}{0.45\textwidth}
		\centering
		\vspace{23pt}
		\begin{tikzpicture}
			\foreach \y in {0}{
				\foreach \x in {0,0.8,1.6,2.4,3.2,4}{
					\draw (\x, \y)[fill]circle[radius=1.0mm]--(\x+0.8,\y);
					
				}
			};
			\draw(4.8,0)[fill]circle[radius=1.0mm];
			\draw(2.4,0.8)[fill]circle[radius=1.0mm]--(2.4,0);
		\end{tikzpicture}
		\caption*{$\tilde{E}_7$}
		\label{fig:4.2}
	\end{subfigure}
	\begin{subfigure}{0.45\textwidth}
		\centering
		\vspace{23pt}
		\begin{tikzpicture}
			\foreach \y in {0}{
				\foreach \x in {0,0.8,1.6,2.4,3.2,4,4.8}{
					\draw (\x, \y)[fill]circle[radius=1.0mm]--(\x+0.8,\y);
					
				}
			};
			\draw(5.6,0)[fill]circle[radius=1.0mm];
			\draw(1.6,0.8)[fill]circle[radius=1.0mm]--(1.6,0);
		\end{tikzpicture}
		\caption*{$\tilde{E}_8$}
		\label{fig:4.3}
	\end{subfigure}
	\label{fig:4}
	\caption{The graphs $\tilde{W}_{n}$, $\tilde{E}_6,\tilde{E}_7,\tilde{E}_8$}
\end{figure}
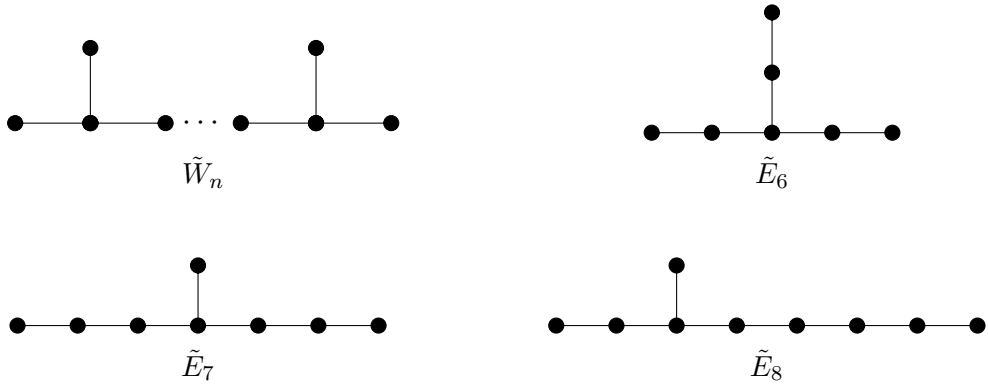
\begin{lem}\label{lem1}\cite{HLi,HLZ}
	Let $n$ and $\psi$ be positive integers with $ \psi>\lceil 2n/3\rceil $.	If $G$ attains the minimum spectral radius  in $\mathcal{G}_{n,\psi}$, then $G$ is a tree.
\end{lem}

Let $f(t)=t(t^2-1), (t\geq 2)$  and $B_1(t;a,b,c;p,q,r)=D_1(t;a,b,c;p,q,r)+E_1$, where $$D_1(t;a,b,c;p,q,r)=\operatorname{diag}((p+a+1)t-{a}/{t},(q+b+2)t-{b}/{t},(r+c+1)t-{c}/{t})$$ and $$E_1=
\begin{bmatrix}
	0 & 1 & 0 \\
	1 & 0 & 1 \\
	0 & 1 & 0
\end{bmatrix}.$$
For convenience, we will also write $B_1(t;a,b,c;p,q,r)$ and $D_1(t;a,b,c;p,q,r)$ as $B_1(t)$ and $D_1(t)$ when the values of $a,b,c,p,q,r$ are clear from the context.
\begin{lem}\label{14}
Let $\rho$ be the spectral radius of a graph $G=G(a,b,c;p,q,r)$ of order $n\ge 14$. Then  $\rho>2$ and
$$\rho(\rho^2-1)=\lambda_1
\left(B_1(\rho)
\right).$$
Moreover, $\rho$ is the maximum positive real root of $f(t)=\lambda_1(B_1(t))$
\end{lem}
\begin{proof}
By Lemma \ref{lemma6} and Lemma \ref{lemma7}, we have $\rho>2$.
For convenience, let $x_i,y_i,z_i,w_i$, with $i=1,2,3$ and $x_{ij}$, with $(i,j)\in\mathcal{I}=\{(1,2),(2,1),(2,3),(3,2)\}$ be the components of the Perron vector $X$ of the graph $G=G(a,b,c;p,q,r)$; see Figure \ref{fig:11}. 
\begin{figure}[H]
	\centering
	\begin{subfigure}{.48\textwidth}
		\centering
		\begin{tikzpicture}
			\draw(0,0)[fill]circle[radius=1.0mm]--(-0.5,1)[fill]circle[radius=1.0mm];
			\draw(1,0)[fill]circle[radius=1.0mm];
			\draw(2,0)[fill]circle[radius=1.0mm];
			\draw(3,0)[fill]circle[radius=1.0mm]--(2.5,1)[fill]circle[radius=1.0mm];
			\draw(3,0)[fill]circle[radius=1.0mm]--(3.5,1)[fill]circle[radius=1.0mm];
			\draw(4,0)[fill]circle[radius=1.0mm];
			\draw(5,0)[fill]circle[radius=1.0mm];
			\draw(6,0)[fill]circle[radius=1.0mm]--(6.5,1)[fill]circle[radius=1.0mm];
			\draw(6,0)[fill]circle[radius=1.0mm]--(5.5,1)[fill]circle[radius=1.0mm];
			\draw(0,0)--(1,0);
			\draw (1,0)--(2,0);
			\draw (2,0)--(3,0);
			\draw (3,0)--(4,0);
			\draw (4,0)--(5,0);
			\draw (5,0)--(6,0);
			\draw(-0.4,-0.8)[fill]circle[radius=1.0mm];\draw(0.4,-0.8)[fill]circle[radius=1.0mm];\draw(-0.4,-1.6)[fill]circle[radius=1.0mm];\draw(0.4,-1.6)[fill]circle[radius=1.0mm];
			\draw(2.6,-0.8)[fill]circle[radius=1.0mm];\draw(3.4,-0.8)[fill]circle[radius=1.0mm];\draw(2.6,-1.6)[fill]circle[radius=1.0mm];\draw(3.4,-1.6)[fill]circle[radius=1.0mm];
			\draw(5.6,-0.8)[fill]circle[radius=1.0mm];\draw(6.4,-0.8)[fill]circle[radius=1.0mm];\draw(5.6,-1.6)[fill]circle[radius=1.0mm];\draw(6.4,-1.6)[fill]circle[radius=1.0mm];
			\draw(0,0)--(-0.4,-0.8);\draw (0,0)--(0.4,-0.8);\draw (3.0,0)--(2.6,-0.8);\draw (3,0)--(3.4,-0.8);\draw (6.0,0)--(5.6,-0.8);\draw (6,0)--(6.4,-0.8);
			\draw(-0.4,-0.8)--(-0.4,-1.6);\draw (0.4,-0.8)--(0.4,-1.6);\draw (2.6,-0.8)--(2.6,-1.6);\draw (3.4,-0.8)--(3.4,-1.6);\draw (5.6,-0.8)--(5.6,-1.6);\draw (6.4,-0.8)--(6.4,-1.6);
			\node at(0,-1.6) {$\dots$};\node at(3,-1.6) {$\dots$};\node at(6,-1.6) {$\dots$};\node[below] at(0,-1.6) {$p$};\node[below] at(3,-1.6) {$q$};\node[below] at(6,-1.6) {$r$};
			\draw(0.5,1)[fill]circle[radius=1.0mm];
			\draw(0,0)--(0.5,1);
            \node at(0.5,1)[right] {$w_1$};\node at(3.5,1)[right] {$w_2$};\node at(6.5,1)[right] {$w_3$};
            \node at(0.4,-0.8)[right] {$y_1$};\node at(3.4,-0.8)[right] {$y_2$};\node at(6.4,-0.8)[right] {$y_3$};
            \node at(0.4,-1.6)[right] {$z_1$};\node at(3.4,-1.6)[right] {$z_2$};\node at(6.4,-1.6)[right] {$z_3$};
            \node at(0,0)[left] {$x_1$};\node at(3.4,0)[above] {$x_2$};\node at(6,0)[right] {$x_3$};
            \node at(1,0)[below] {$x_{12}$};\node at (2,0)[below] {$x_{21}$};\node at(4,0)[below] {$x_{23}$};\node at(5,0)[below] {$x_{32}$};
			\node at(0,1)[above=2pt] {$a$};
			\node at(0,1){$\cdots$};
			\node at(3,1)[above=2pt] {$b$};
			\node at(3,1){$\cdots$};
			\node at(6,1)[above=2pt] {$c$};
			\node at(6,1){$\cdots$};
		\end{tikzpicture}
	\end{subfigure}
	\caption{$G(a,b,c;p,q,r)$ with labeled Perron vector}
	\label{fig:11}
\end{figure}
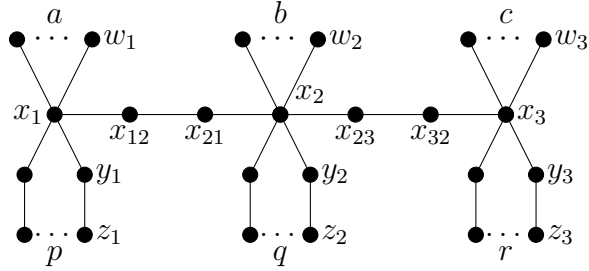
\vspace{-20pt}
By symmetry of the Perron vector's components and   $\rho X=A_GX$, we have
$$\begin{cases}
	\rho y_i=x_i+z_i,\\
	\rho z_i=y_i,\\
	\rho x_{ij}=x_{i}+x_{ji},\\
	\rho w_i=x_i,
\end{cases}$$
for $i=1,2,3$ and $(i,j)\in\mathcal{I}$,
 \text{which lead to} \quad
\begin{equation}\label{eq3.1}
\left\{
\begin{array}{l}
	y_i=\frac{\rho}{\rho^2-1}x_i,\\
	z_i=\frac{1}{\rho^2-1}x_i,\\
	x_{ij}=\frac{\rho x_i+x_j}{\rho^2-1},\\
	 w_i=\frac{x_i}{\rho}.
\end{array}\right.
\end{equation}
Considering the equations in $\rho X=A_GX$ corresponding to the components $x_i, i=1,2,3$, we have
$$\begin{cases}
		\rho x_1 = py_1 + x_{12} + aw_1 = \dfrac{p\rho}{\rho^2-1}x_1 + \dfrac{\rho x_1+x_2}{\rho^2-1} + \dfrac{ax_1}{\rho}, \\
		\rho x_2 = qy_2 + x_{21} + x_{23} + bw_2 = \dfrac{q\rho}{\rho^2-1}x_2 + \dfrac{\rho x_2+x_1}{\rho^2-1} + \dfrac{\rho x_2+x_3}{\rho^2-1} + \dfrac{bx_2}{\rho}, \\
		\rho x_3 = ry_3 + x_{32} + cw_3 = \dfrac{r\rho}{\rho^2-1}x_3 + \dfrac{\rho x_3+x_2}{\rho^2-1} + \dfrac{cx_3}{\rho},
	\end{cases}$$
	which lead to
	$$\rho(\rho^2 - 1)
	\begin{bmatrix}
		x_1 \\
		x_2 \\
		x_3
	\end{bmatrix}
	=
	\begin{bmatrix}
		(p+a+1)\rho - \dfrac{a}{\rho} & 1 & 0 \\
		1 & (q+b+2)\rho - \dfrac{b}{\rho} & 1 \\
		0 & 1 & (r+c+1)\rho - \dfrac{c}{\rho}
	\end{bmatrix}
	\begin{bmatrix}
		x_1 \\
		x_2 \\
		x_3
	\end{bmatrix}.
$$
Let $X'=(x_1,x_2,x_3)^T$.
Since $f(\rho)X'=B_1(\rho)X'$,  $X'$ is an eigenvector of $B_1(\rho)$.
 Recalling that $X$ is the Perron vector of $G(a,b,c;p,q,r)$, we have $X>0$, which leads to $X'>0$.
By Lemma \ref{PF}, $X'$ is an eigenvector corresponding to the eigenvalue $\lambda_1(B_1(\rho))$. Thus, we have $f(\rho)=\lambda_1(B_1(\rho))$.

Now we prove that $\rho$ is the maximum positive real root of $f(t)=\lambda_1(B_1(t))$.
If $\tilde{\rho}>\rho$ satisfies $f(\tilde{\rho})=\lambda_1(B_1(\tilde{\rho}))$,
by Lemma \ref{PF}, there exists $\tilde{X}'=(\tilde{x_1},\tilde{x_2},\tilde{x_3})^T>0$ satisfying $f(\tilde{\rho})\tilde{X}'=\lambda_1(B_1(\tilde{\rho}))\tilde{X}'$.
Similarly as in equation (\ref{eq3.1}), the vector $\tilde{X}'$ can be extended to form an eigenvector $\tilde{X}$ of $A_G$ corresponding to the eigenvalue $\tilde{\rho}$, which contradicts the assumption that $\rho$ is maximum eigenvalue of $A_G$.
\end{proof}

\begin{corollary}\label{15}
  Let $n=6m+l$ with $m\geq 2$ and $0\leq l\leq 5$. Then $\rho^2(G_{m,l})<m+3$.
\end{corollary}
\begin{proof}
 Take $(a,b,c,p,q,r)=(0,0,0,m,m-1,m)$. Then we have   $$G(a,b,c;p,q,r)=G_{m,5} \quad\text{ and }\quad
 B_1(t)=(m+1)tI+E_1.$$  
Let $\rho=\rho(G_{m,5})$. Then by 
  Lemma \ref{14}, we have  $$\rho(\rho^2-1)=\lambda_1(B_1(\rho))= (m+1)\rho+\sqrt{2}<(m+2)\rho,$$ which leads to $\rho^2< m+3$. Since $G_{m,l}\subseteq G_{m,5}$, $\rho^2(G_{m,l})<m+3$.
\end{proof}
\begin{lem}\label{3}
	Let $f(t)$, $h_1(t), h_2(t)$ be  real valued functions for $t\in [a,+\infty)$. Suppose $\rho_i\in [a,+\infty)$ is the maximum root of $f(t)=h_i(t), i=1,2$. If $h_1(t)\geq h_2(t)$ for all $t\in [a,+\infty)$ and $\lim\limits_{n\to\infty}(f(t)-h_i(t))=+\infty$ for $i=1,2$, then $\rho_1\geq \rho_2$.
\end{lem}

\begin{proof}
	To the contrary, we suppose $\rho_1<\rho_2$. Since $\rho_1$ is the maximum root of $f(t)=h_1(t)$ and $\lim\limits_{n\rightarrow \infty}(f(t)-h_1(t))=+\infty$,
	we have $f(\rho_1)=h_1(\rho_1)$ and $f(\rho_2)>h_1(\rho_2)\geq h_2(\rho_2)$, which contradicts $f(\rho_2)=h_2(\rho_2)$.
\end{proof}
\section{Proof of Theorem \ref{theo1}}

Denoted by $H(a,b,c;p,q,r)$ the following graph obtained from the graph $W_5$ by attaching $a$ leaves and $p$ edges, $b$ leaves and $q$ edges, $c$ leaves and $r$ edges to the three leaves of $W_5$; see Figure \ref{fig}.
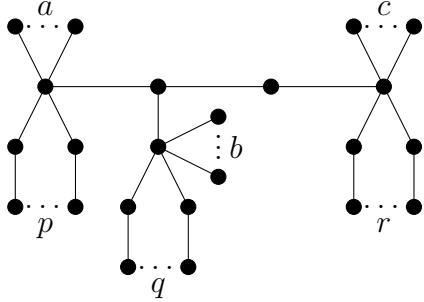
\begin{figure}[H]
	\centering
	\begin{tikzpicture}
		\draw(0,0)[fill]circle[radius=1.0mm];\draw(1.5,0)[fill]circle[radius=1.0mm];\draw(3,0)[fill]circle[radius=1.0mm];\draw(4.5,0)[fill]circle[radius=1.0mm];	
		\draw(0,0)--(1.5,0);\draw (1.5,0)--(3,0);\draw (3,0)--(4.5,0);
		\draw(-0.4,0.8)[fill]circle[radius=1.0mm];\draw(0.4,0.8)[fill]circle[radius=1.0mm];\draw(2.3,-0.4)[fill]circle[radius=1.0mm];\draw(2.3,-1.2)[fill]circle[radius=1.0mm];
		\draw(4.1,0.8)[fill]circle[radius=1.0mm];\draw(4.9,0.8)[fill]circle[radius=1.0mm];
		\draw(-0.4,-0.8)[fill]circle[radius=1.0mm];\draw(0.4,-0.8)[fill]circle[radius=1.0mm];\draw(-0.4,-1.6)[fill]circle[radius=1.0mm];\draw(0.4,-1.6)[fill]circle[radius=1.0mm];
		\draw(1.1,-1.6)[fill]circle[radius=1.0mm];\draw(1.9,-1.6)[fill]circle[radius=1.0mm];\draw(1.1,-2.4)[fill]circle[radius=1.0mm];\draw(1.9,-2.4)[fill]circle[radius=1.0mm];
		\draw(4.1,-0.8)[fill]circle[radius=1.0mm];\draw(4.9,-0.8)[fill]circle[radius=1.0mm];\draw(4.1,-1.6)[fill]circle[radius=1.0mm];\draw(4.9,-1.6)[fill]circle[radius=1.0mm];
		\draw(0,0)--(-0.4,-0.8);\draw (0,0)--(0.4,-0.8);\draw (1.5,-0.8)--(1.1,-1.6);\draw (1.5,-0.8)--(1.9,-1.6);\draw (4.5,0)--(4.1,-0.8);\draw (4.5,0)--(4.9,-0.8);
		\draw(0,0)--(-0.4,0.8);\draw (0,0)--(0.4,0.8);\draw (4.5,0)--(4.1,0.8);\draw (4.5,0)--(4.9,0.8);\draw (1.5,-0.8)--(2.3,-0.4);\draw(1.5,-0.8)--(2.3,-1.2);
		\draw(-0.4,-0.8)--(-0.4,-1.6);\draw (0.4,-0.8)--(0.4,-1.6);\draw (1.1,-1.6)--(1.1,-2.4);\draw (1.9,-1.6)--(1.9,-2.4);\draw (4.1,-0.8)--(4.1,-1.6);\draw (4.9,-0.8)--(4.9,-1.6);
		\node at(0,-1.6) {$\dots$};\node at(1.5,-2.4) {$\dots$};\node at(4.5,-1.6) {$\dots$};\node[below] at(0,-1.6) {$p$};\node[below] at(1.5,-2.4) {$q$};\node[below] at(4.5,-1.6) {$r$};
		\node at(0,0.8) {$\dots$};\node at(2.3,-0.7) {$\vdots$};\node at(4.5,0.8) {$\dots$};\node[above] at(0,0.8) {$a$};\node[right] at(2.3,-0.8) {$b$};\node[above] at(4.5,0.8) {$c$};
		\draw(1.5,-0.8)[fill]circle[radius=1.0mm];\draw(1.5,0)--(1.5,-0.8);
	\end{tikzpicture}
	\caption {$H(a,b,c;p,q,r)$}
	\label{fig}
\end{figure}
\textbf{Definition 1.} The {\it skeleton} of $G$, denoted by $S[G;v_1,v_2,\ldots,v_k]$, is obtained from $G$ by deleting all the components of $G-\{v_1,v_2,\ldots,v_k\}$ that are connected to exactly one vertex of $v_1,v_2,\ldots,v_k$.

Suppose $D$ is a maximum dissociation set of $G$ such that $G[D]$ consists of $\gamma$ isolated vertices $u_1,\dots, u_{\gamma}$ and $\tau$ disjoint edges $u^1_1u^2_1,\ldots, u^1_{\tau}u^2_{\tau}$. Since $G$ is connected, any isolated vertex or edge of $G[D]$ is adjacent to at least one vertex of $V(G)\backslash D$.

\textbf{Definition 2.} The {\it generated hypergraph} of $G=(V,E)$ with respect to $D$, denoted $H[G,D]$, is the hypergraph  $(V',E',\phi)$, where $V'=V(G)\setminus D$ and $E'$ consists of the following three types of hyperedges:\\
\indent (1) the edges in $E$ with both ends in $V'$;\\
\indent (2)  $N(u_i)\cap V' $ with $|N(u_i)\cap V'|\ge 2$, $i=1,2,\dots,\gamma$;\\
\indent (3)  $N(\{u^1_j,u^2_j\})\cap V'$ with $|N(\{u^1_j,u^2_j\})\cap V'|\ge 2$, $j=1,2,\dots,\tau$.

Since $G\in \mathcal{G}_{n,n-3}$, we have $|V(G)\setminus D|=3$. Denote by  $V(G)\setminus D=\{v_1,v_2,v_3\}$. We consider the structure of $H[G,D]$.

\textbf{Claim 1.} {\it Two distinct hyperedges $e,f\in E'$ are incident to at most one common vertex.}

{\it Proof of Claim 1.} Suppose there exist two edges $e,f\in E'$ such that $e$ and $f$ are both incident to two distinct vertices $v_1,v_2$. Then $v_1ev_2fv_1$ is a cycle in $H[G,D]$, which leads to a cycle in $G$. This contradicts  the conclusion that $G$ is a tree.    \qed

\textbf{Claim 2.}  {\it  $H[G, D]$ is not isomorphic to the 2-uniform graph $C_3$.}

{\it Proof of Claim 2.} To the contrary, suppose $e$ is incident to $v_i$ and $v_j$, $f$ is incident to $v_j$ and $v_k$ and $g$ is incident to $v_i$ and $v_k$. Then $v_iev_jfv_kgv_i$ is a cycle in $H[G,D]$, which leads to a cycle in $G$, a contradiction.  \qed

\textbf{Claim 3.}  {\it If $G\in \mathcal{G}_{n,n-3}$,
then one of the following holds.
\begin{itemize}
\item[(i)] There exists $H(a,b,c;p,q,r)$ with $a,b,c\in \{0,1\}$ and $q\geq 1$ such that $\rho(G)\geq \rho(H(a,b,c;\\p,q,r))$;
 \item[(ii)]   there exists $G(a,b,c;p,q,r)$ with $a,b,c\in \{0,1\}$ such that $\rho(G)> \rho(G(a,b,c;p,q,r))$;
   \item[(iii)]    $G\cong G(a,b,c;p,q,r)$ with $a,b,c\in \{0,1\}$.
   \end{itemize}}

{\it Proof of Claim 3.}
Since all components of $G[D]$ are connected to at least one of $v_1,v_2, v_3$ and  $G[D]$ has at least $\lceil (n-3)/2\rceil$ components, we have $$d(v_1)+d(v_2)+d(v_3)\geq \lceil {(n-3)}/{2}\rceil\geq {(n-3)}/{2}.$$ If $d(v_i)\leq 2$ for  some $i\in\{1,2,3\}$, then   there  exists a vertex $v_j\in\{v_1,v_2,v_3\}\setminus \{v_i\}$ with $d(v_j)\geq  ({n-7})/{4}$,
which means $G$ contains a subgraph $S_t$ with $t=({n-7})/{4}$. By Corollary \ref{15}, we have $\rho^2(G_{m,l})< m+3\leq \lfloor {n}/{6}\rfloor+3$, which implies $\rho(G_{m,l})<\rho(G)$ and the claim holds. So we need only to consider the case $d(v_i)\geq 3$ for $i=1,2,3$. We distinguish two cases.

\textbf{Case 1.} There exists an edge $e\in E'$ incident to all the vertices of $V'$.
By Claim 1, $H[G,D]$ contains exactly one edge $e$, which implies that $S[G;v_1,v_2,v_3]$ must be isomorphic to one of the following $S_1,S_2,S_3$.

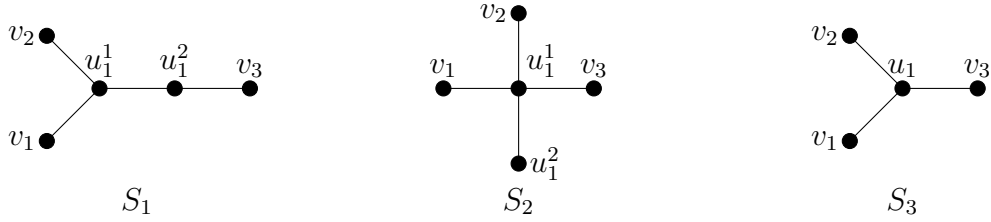
\begin{figure}[H]
	\centering
    \begin{subfigure}{.3\textwidth}
	\centering
	\begin{tikzpicture}
	\draw(-0.7,0.7)[fill]circle[radius=1.0mm];\draw(-0.7,-0.7)[fill]circle[radius=1.0mm];\draw(0,0)[fill]circle[radius=1.0mm];
    \draw(1,0)[fill]circle[radius=1.0mm];\draw(2,0)[fill]circle[radius=1.0mm];
    \draw(-0.7,0.7)--(0,0);\draw(-0.7,-0.7)--(0,0);\draw(0,0)--(1,0);\draw(1,0)--(2,0);
    \node [left] at (-0.7,-0.7) {$v_1$};\node [left] at (-0.7,0.7) {$v_2$};\node [above] at (2,0) {$v_3$};\node [above] at (0,0) {$u^1_1$};\node [above] at (1,0) {$u^2_1$};
    \node at(0.5,-1.5) {$S_1$};
	\end{tikzpicture}
	\end{subfigure}
    \begin{subfigure}{.3\textwidth}
	\centering
	\begin{tikzpicture}
	\draw(-1,0)[fill]circle[radius=1.0mm];\draw(0,-1)[fill]circle[radius=1.0mm];\draw(0,1)[fill]circle[radius=1.0mm];\draw(0,0)[fill]circle[radius=1.0mm];\draw(1,0)[fill]circle[radius=1.0mm];
    \draw(-1,0)--(0,0);\draw(0,-1)--(0,0);\draw(0,1)--(0,0);\draw(1,0)--(0,0);
    \node [above] at (-1,0) {$v_1$};\node [left] at (0,1) {$v_2$};\node [above] at (1,0) {$v_3$};\node [above] at (0.3,0) {$u^1_1$};\node [right] at (0,-1) {$u^2_1$};
    \node at(0,-1.5) {$S_2$};
	\end{tikzpicture}
	\end{subfigure}
	\begin{subfigure}{.3\textwidth}
	\centering
	\begin{tikzpicture}
	\draw(-0.7,0.7)[fill]circle[radius=1.0mm];\draw(-0.7,-0.7)[fill]circle[radius=1.0mm];\draw(0,0)[fill]circle[radius=1.0mm];\draw(1,0)[fill]circle[radius=1.0mm];
    \draw(-0.7,0.7)--(0,0);\draw(-0.7,-0.7)--(0,0);\draw(0,0)--(1,0);
    \node [left] at (-0.7,-0.7) {$v_1$};\node [left] at (-0.7,0.7) {$v_2$};\node [above] at (1,0) {$v_3$};\node [above] at (0,0) {$u_1$};
    \node at(0,-1.5) {$S_3$};
	\end{tikzpicture}
	\end{subfigure}
\end{figure}

For $S[G;v_1,v_2,v_3]\cong S_2$, we denote by $G'$ the graph obtained from $G$ by subdividing the edge $u^1_1v_1$ and deleting the vertex $u_1^2$;
for $S[G;v_1,v_2,v_3]\cong S_3$, we denote by $G'$ the graph obtained from $G$ by subdividing the edge $u_1v_1$ and deleting a leaf not adjacent to $v_2$. In these two cases we both have $S[G';v_1,v_2,v_3]\cong S_1$, which leads to $G'\cong H(a,b,c;p,q,r)$, since $V\setminus\{v_1,v_2,v_3\}$ forms a dissociation set of $G'$.
Moreover, by Lemma \ref{lemma3} and Lemma \ref{lemma11}, we have $\rho(G')<\rho(G)$.

Noticing that $d_{G'}(v_2)=d_G(v_2)\geq 3$, we have $q+b\geq 2$. By Lemma \ref{lemma2}, the minimum spectral radius of graphs isomorphic to $H(a,b,c;p,q,r)$ of order $n$ is attained at $a,b,c\in\{0,1\}$.
Therefore, we can always find a graph $G'\cong H(a,b,c;p,q,r)$ with $b\in \{0,1\}$ and $q\ge 1$ such that $\rho(G)\ge \rho(G')$.

\textbf{Case 2.}  $H[G,D]$ is 2-uniform. We first prove that $\rho(G)\geq \rho(G(a,b,c;p,q,r))$ for some nonnegative integers $a,b,c,p,q,r$.

Since $G$ is a connected graph, then $H[G,D]$ is a connected graph. By Claim 2, we have $H[G,D]\cong P_3$. Without loss of generality,
we may assume the two edges in $H[G,D]$ are $\{v_1,v_2\}$ and $\{v_2,v_3\}$. Then both $S[G;v_1,v_2]$ and $S[G;v_2,v_3]$  must be isomorphic to one of the following $S',S'_1,S'_2,S'_3$. Notice that $G\cong G(a,b,c;p,q,r) $ for some nonnegative integers $a,b,c,p,q,r$ if and only if $H[G,D]$ is 2-uniform, $H[G,D]\cong P_3$, $S[G;v_1,v_2]\cong S'$ and $S[G;v_2,v_3]\cong S'$.

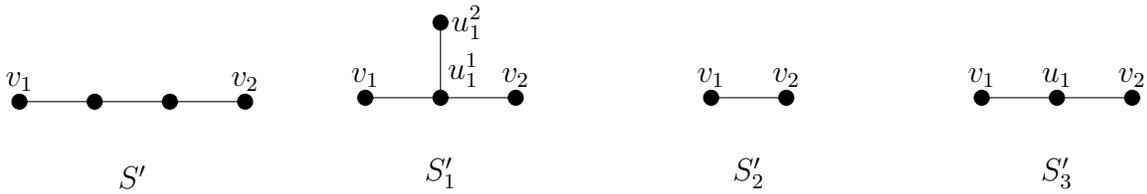
\begin{figure}[H]
	\centering
    \begin{subfigure}{.24\textwidth}
	\centering
    \vspace{21pt}
	\begin{tikzpicture}
    \draw(-1,0)[fill]circle[radius=1.0mm];\draw(0,0)[fill]circle[radius=1.0mm];\draw(1,0)[fill]circle[radius=1.0mm];\draw(2,0)[fill]circle[radius=1.0mm];
    \draw(-1,0)--(0,0); \draw(0,0)--(1,0); \draw(1,0)--(2,0);
    \node [above] at (-1,0) {$v_1$}; \node [above] at (2,0) {$v_2$};
    \node at (0.5,-1) {$S'$};
	\end{tikzpicture}
	\end{subfigure}
    \begin{subfigure}{.24\textwidth}
	\centering
	\begin{tikzpicture}
    \draw(-1,0)[fill]circle[radius=1.0mm];\draw(0,0)[fill]circle[radius=1.0mm];\draw(1,0)[fill]circle[radius=1.0mm];\draw(0,1)[fill]circle[radius=1.0mm];
    \draw(-1,0)--(0,0); \draw(1,0)--(0,0);\draw(0,0)--(0,1);\node [above] at (-1,0) {$v_1$}; \node [above] at (1,0) {$v_2$};
    \node [above] at (0.3,0) {$u^1_1$}; \node [right] at (0,1) {$u^2_1$};
    \node at (0,-1) {$S'_1$};
	\end{tikzpicture}
	\end{subfigure}
	\begin{subfigure}{.24\textwidth}
	\centering
    \vspace{21pt}
	\begin{tikzpicture}
    \draw(-1,0)[fill]circle[radius=1.0mm];\draw(0,0)[fill]circle[radius=1.0mm];
    \draw(-1,0)--(0,0); \node [above] at (-1,0) {$v_1$}; \node [above] at (0,0) {$v_2$};
    \node at (-0.5,-1) {$S'_2$};
	\end{tikzpicture}
	\end{subfigure}
    \begin{subfigure}{.24\textwidth}
	\centering
    \vspace{21pt}
	\begin{tikzpicture}
	\draw(-1,0)[fill]circle[radius=1.0mm];\draw(0,0)[fill]circle[radius=1.0mm];\draw(1,0)[fill]circle[radius=1.0mm];
    \draw(-1,0)--(0,0);\draw(1,0)--(0,0);\node [above] at (-1,0) {$v_1$}; \node [above] at (1,0) {$v_2$}; \node [above] at (0,0) {$u_1$};
    \node at (0,-1) {$S'_3$};
	\end{tikzpicture}
	\end{subfigure}
\end{figure}

 If $S[G;v_1,v_2]\cong S'_1$, let $G'$ be the graph obtained from $G$ by subdividing the edge $v_1u^1_1$ and deleting $u^2_1$;
if $S[G;v_1,v_2]\cong S'_2$, let $G'$ be the graph obtained from $G$ by subdividing the edge $v_1v_2$ twice and deleting two leaves neither adjacent to $v_2$ nor $v_3$; if
 $S[G;v_1,v_2]\cong S'_3$, let $G'$ be the graph obtained from $G$ by subdividing the edge $v_1u_1$ and deleting one leaf neither adjacent to $v_2$ nor $v_3$. By Lemma \ref{lemma3} and \ref{lem1}, we get $\rho(G')< \rho(G)$. Moreover, we have
 \begin{itemize}
 \item[(i)] $S[G';v_1,v_2]\cong S'$;
 \item[(ii)] $S[G';v_2,v_3]\cong S[G;v_2,v_3]$;
 \item[(iii)] $V-\{v_1,v_2,v_3\}$ is a dissociation set of $G'$;
 \item[(iv)] $d_{G'}(v_i)=d_{G}(v_i)\geq 3,i=2,3$.
 \end{itemize}
Now we consider the structure of $S[G';v_2,v_3]$. If  $S[G';v_2,v_3]\cong S'$, then $G'\cong G(a,b,c;p,q,r) $ for some nonnegative integers $a,b,c,p,q,r$.  If $S[G';v_2,v_3]\not\cong S'$, then similarly as above,   we can always find a graph $G''$ such that $\rho(G'')<\rho(G')$,  $S[G'';v_1,v_2]\cong S'$, $S[G'';v_2,v_3]\cong S'$ and $V-\{v_1,v_2,v_3\}$ is a dissociation set of $G''$,
which leads to $G''\cong G(a,b,c;p,q,r)$ for some nonnegative integers $a,b,c,p,q,r$.

Finally, applying Lemma \ref{lemma2}, we have either (ii) or (iii). \qed
 \par
{\bf Claim 4.} {\it If $G\in \{G(a,b,c;p,q,r): a,b,c,p,q,r\in  \mathbb{N}\}\cup \{H(a,b,c;p,q,r):a,b,c,p,r\in  \mathbb{N}, q\in \mathbb{N}^+\}$ is an $n$-vertex graph with minimum spectral radius, then $G\cong G(a,b,c;p,q,r)$ such that
\begin{itemize}
\item [(i)]\vskip-0.4cm   $a,b,c\in\{0,1\}$;
\item [(ii)] $p+a,q+b+1,r+c$ differ by at most 1 from each other.
\end{itemize}}

{\it Proof of Claim 4.} We distinguish two cases.

\textbf{Case 1.} $G\cong G(a,b,c;p,q,r)$ with $a,b,c\in \{0,1\}$. Let $f(t), B_1(t)$, $D_1(t)$ and $E_1$ be defined as in Section 2.
By Lemma \ref{14}, $\rho=\rho(G)$ is the maximum positive real root of $f(t)-\lambda_1(B_1(t))=0$ and $f(t)=O(t^3)$, $\lambda_1(B_1(t))=O(t)$ as $t\rightarrow +\infty$.
By Lemma \ref{3}, if $\rho(B_1(t;a,b,c;p,q,r))\geq \rho(B_1(t;a',b',c';p',q',r'))$, then $$\rho(G(a,b,c;p,q,r))\geq \rho(G(a',b',c';p',q',r')).$$
Since $D_1(t),E_1$ are symmetric nonnegative matrices, by Lemma \ref{W_ieq}, we have
\begin{equation*}
\begin{aligned}
	&\max\{(p+a+1),(q+b+2),(r+c+1)\}t-\frac{1}{t} \leq \lambda_1(B_1(t)) \\&\leq \max\{(p+a+1),(q+b+2),(r+c+1)\}t+\sqrt{2}.
\end{aligned}
\end{equation*}
Now we can conclude that  if $G\in \{G(a,b,c;p,q,r):a,b,c,p,q,r \in \mathbb{N}\}$ attaining the minimum spectral radius, then $p+a,q+b+1,r+c$ differ by at most 1 from each other.  \par

\textbf{Case 2.} $G\cong H(a,b,c;p,q,r)$ with $a,b,c\in\{0,1\}$ and $q\geq 1$.

\begin{figure}[H]
	\centering
	\begin{tikzpicture}
	\draw(0,0)[fill]circle[radius=1.0mm];\draw(1.5,0)[fill]circle[radius=1.0mm];\draw(3,0)[fill]circle[radius=1.0mm];\draw(4.5,0)[fill]circle[radius=1.0mm];	
    \draw(0,0)--(1.5,0);\draw (1.5,0)--(3,0);\draw (3,0)--(4.5,0);
    \draw(-0.4,0.8)[fill]circle[radius=1.0mm];\draw(0.4,0.8)[fill]circle[radius=1.0mm];\draw(2.3,-0.4)[fill]circle[radius=1.0mm];\draw(2.3,-1.2)[fill]circle[radius=1.0mm];
    \draw(4.1,0.8)[fill]circle[radius=1.0mm];\draw(4.9,0.8)[fill]circle[radius=1.0mm];
    \draw(-0.4,-0.8)[fill]circle[radius=1.0mm];\draw(0.4,-0.8)[fill]circle[radius=1.0mm];\draw(-0.4,-1.6)[fill]circle[radius=1.0mm];\draw(0.4,-1.6)[fill]circle[radius=1.0mm];
    \draw(1.1,-1.6)[fill]circle[radius=1.0mm];\draw(1.9,-1.6)[fill]circle[radius=1.0mm];\draw(1.1,-2.4)[fill]circle[radius=1.0mm];\draw(1.9,-2.4)[fill]circle[radius=1.0mm];
    \draw(4.1,-0.8)[fill]circle[radius=1.0mm];\draw(4.9,-0.8)[fill]circle[radius=1.0mm];\draw(4.1,-1.6)[fill]circle[radius=1.0mm];\draw(4.9,-1.6)[fill]circle[radius=1.0mm];
    \draw(0,0)--(-0.4,-0.8);\draw (0,0)--(0.4,-0.8);\draw (1.5,-0.8)--(1.1,-1.6);\draw (1.5,-0.8)--(1.9,-1.6);\draw (4.5,0)--(4.1,-0.8);\draw (4.5,0)--(4.9,-0.8);
    \draw(0,0)--(-0.4,0.8);\draw (0,0)--(0.4,0.8);\draw (4.5,0)--(4.1,0.8);\draw (4.5,0)--(4.9,0.8);\draw (1.5,-0.8)--(2.3,-0.4);\draw(1.5,-0.8)--(2.3,-1.2);
    \draw(-0.4,-0.8)--(-0.4,-1.6);\draw (0.4,-0.8)--(0.4,-1.6);\draw (1.1,-1.6)--(1.1,-2.4);\draw (1.9,-1.6)--(1.9,-2.4);\draw (4.1,-0.8)--(4.1,-1.6);\draw (4.9,-0.8)--(4.9,-1.6);
    \node at(0,-1.6) {$\dots$};\node at(1.5,-2.4) {$\dots$};\node at(4.5,-1.6) {$\dots$};\node[below] at(0,-1.6) {$p$};\node[below] at(1.5,-2.4) {$q$};\node[below] at(4.5,-1.6) {$r$};
    \node at(0,0.8) {$\dots$};\node at(2.3,-0.7) {$\vdots$};\node at(4.5,0.8) {$\dots$};\node[above] at(0,0.8) {$a$};\node[right] at(2.3,-0.8) {$b$};\node[above] at(4.5,0.8) {$c$};
    \draw(1.5,-0.8)[fill]circle[radius=1.0mm];\draw(1.5,0)--(1.5,-0.8);
    \node at(0,0)[left] {$x_1$};\node at(1.5,0)[above] {$x'_{12}$};\node at(3,0)[above] {$x'_3$};\node at(4.5,0)[right] {$x_3$};\node at(1.5,-0.8)[left] {$x_2$};
    \node at(-0.4,0.8)[left] {$w_1$};\node at(-0.4,-0.8)[left] {$y_1$};\node at(-0.4,-1.6)[left] {$z_1$};\node at(1.9,-1.6)[right] {$y_2$};\node at(1.9,-2.4)[right] {$z_2$};
    \node at(2.3,-1.3)[right] {$w_2$};\node at(4.9,0.8)[right] {$w_3$};\node at(4.9,-0.8)[right] {$y_3$};\node at(4.9,-1.6)[right] {$z_3$};
	\end{tikzpicture}
    \caption{$H(a,b,c;p,q,r)$ with labeled Perron vector}
	\label{fig:12}
\end{figure}
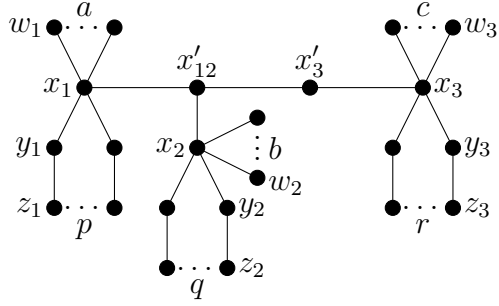
\vspace{-0.3cm}
For convenience, let $x_i,y_i,z_i,w_i,i=1,2,3$ and $x'_{S}, S\in\{\{12\},\{3\}\}
$ be the coordinates of the Perron vector $X$ of the graph $H(a,b,c;p,q,r)$; See Figure \ref{fig:12}. Suppose $\rho(G)=\rho$. By symmetry of the Perron vector's components and   $\rho X=A_GX$, we have
\begin{equation*}\label{eq3.2}
\left\{
\begin{array}{l}
\rho y_i=x_i+z_i,\\
\rho z_i=y_i,\\
\rho x'_{S}=\sum_{i\in S}x_i+x'_{[3]\setminus S},\\
\rho w_i=x_i,
\end{array}
\right.
\end{equation*}
for $i=1,2,3$, $(i,j)\in\mathcal{I}$ and $S\in\{\{12\},\{3\}\}$,
\text{which lead to} \quad
$$\left\{
\begin{array}{l}
	y_i=\frac{\rho}{\rho^2-1}x_i,\\
	z_i=\frac{1}{\rho^2-1}x_i,\\
	x'_{S}=\frac{\rho (\sum_{i\in S}x_i)+\sum_{j\in [3]\setminus S}x_j}{\rho^2-1},\\
	w_i=\frac{x_i}{\rho}.
\end{array}\right.$$
Considering the characteristic equations corresponding to the components  $x_i, i=1,2,3$, we have
$$\begin{cases}
	\rho x_1=py_1+x'_{12}+aw_1=\frac{p\rho}{\rho^2-1}x_1+\frac{\rho (x_1+x_2)+x_3}{\rho^2-1}+\frac{ax_1}{\rho},\\
	\rho x_2=qy_2+x'_{12}+bw_2=\frac{q\rho}{\rho^2-1}x_2+\frac{\rho (x_1+x_2)+x_3}{\rho^2-1}+\frac{bx_2}{\rho},\\
	\rho x_3=ry_3+x'_3+cw_3=\frac{r\rho}{\rho^2-1}x_3+\frac{ x_1+x_2+\rho x_3}{\rho^2-1}+\frac{cx_3}{\rho},
	\end{cases}$$
which lead to
$$\rho(\rho^2-1)
\begin{bmatrix}
x_1 \\
x_2 \\
x_3
\end{bmatrix}
=
\begin{bmatrix}
(p+a+1)\rho-\frac{a}{\rho} & \rho & 1 \\
\rho & (q+b+1)\rho-\frac{b}{\rho} & 1 \\
1 & 1 & (r+c+1)\rho-\frac{c}{\rho}
\end{bmatrix}
\begin{bmatrix}
x_1 \\
x_2 \\
x_3
\end{bmatrix}.
$$

Let $$B_2(t;a,b,c;p,q,r)=
\begin{bmatrix}
(p+a+1)\rho-\frac{a}{\rho} & \rho & 1 \\
\rho & (q+b+1)\rho-\frac{b}{\rho} & 1 \\
1 & 1 & (r+c+1)\rho-\frac{c}{\rho}
\end{bmatrix},$$
which we also write as $B_2(t)$.  Again, let $f(t)=t(t^2-1)$ with $t\geq 2$.
Following an analogous approach to Case 1, we can conclude  that $\rho$ is the maximum positive real root of $f(t)=\lambda_1(B_2(t))$.

Since each entry of $B_2(t;a,b,c;p,q,r)$ is greater than or equal to the corresponding entry of $B_1(t;a,b,c;p,q-1,r)$, we have $\lambda_1(B_2(t;a,b,c;p,q,r))>\lambda_1(B_1(t;a,b,c;p,q-1,r))$, which leads to $\rho(H(a,b,c;p,q,r))>\rho(G(a,b,c;p,q-1,r))$.
Hence, we can conclude that Claim 4 holds. \qed\\
\par By Claim 3 and Claim 4, the minimum spectral radius in $\mathcal{G}_{n,n-3}$ is attained at a graph $G(a,b,c;p,q,r)$ satisfying the conditions of Claim 4. So it suffices to characterize  the graphs $G(a,b,c;p,q,r)$ satisfying  conditions of Claim 4 that attain  the minimum spectral radius.

For any $G\cong G(a,b,c;m+m_1,m+m_2,m+m_3)$, $m_1,m_2,m_3\in\{-3,-2,-1,0,1\}$, we have
  \begin{align}
  	& \left(\begin{array}{ccc}
  	(m+m_1+a+1)t-\frac{a}{t} & 1 & 0 \\
  	1 & (m+m_2+b+2)t-\frac{b}{t} & 1 \\
  	0 & 1 & (m+m_3+c+1)t-\frac{c}{t}
  	\end{array}\right)\nonumber\\&=t (m+1) I+\left(\begin{array}{ccc}
  		(m_1+a)t-\frac{a}{t} & 1 & 0 \\
  		1 & (m_2+b+1)t-\frac{b}{t} & 1 \\
  		0 & 1 & (m_3+c)t-\frac{c}{t}
  	\end{array}\right)\label{equ2}\\&\equiv t (m+1) I+A_{m_1,m_2,m_3}^{a,b,c}\nonumber
  \end{align}
  By Claim 3 and equation (\ref{equ2}), for any  possible $G(a,b,c;m+m_1,m+m_2,m+m_3)$, we compute the characteristic polynomial of the corresponding matrix $A_{m_1,m_2,m_3}^{a,b,c}$ of  $G(a,b,c;m+m_1,m+m_2,m+m_3)$ so that we can compare the spectral radius of $G(a,b,c;m+m_1,m+m_2,m+m_3)$ to identify the graph $G$ attaining the minimum spectral radius.

Recall that $m=\left\lfloor {n}/{6}\right\rfloor$ with $n\ge 39$. Suppose $G\cong G(a,b,c;p,q,r)$ attains the minimum spectral radius. We distinguish six cases.\\
  \textbf{Case 1.} $n=6m$. By Claim 4 we have $$G\in\left\{\begin{array}{ll}G(1,0,0;m-1, m-2, m-1),&G(0,1,0;m-1, m-2, m-1),\\G(0,1,0;m-1, m-3, m),&G(1,0,0;m-2, m-2, m),\\G(1,0,0;m-2, m-1, m-1),&G(1,1,1;m-1, m-3, m-1),\\G(1,1,1;m-2, m-2, m-1)\\\end{array}
  	\right\}.$$
  By  equation (\ref{equ2}),  we have
  \vskip -0.5cm
 \begin{align*}
  	& \operatorname{det}\left(\lambda I-A_{-1,-2,-1}^{1,0,0}\right)=\lambda^3+\left(2 t+\frac{1}{t}\right) \lambda^2+t^2 \lambda-\frac{1}{t}\equiv f_1(\lambda), \\
  	& \operatorname{det}\left(\lambda I-A^{0,1,0}_{-1,-2,-1}\right)=\lambda^3+\left(2 t+\frac{1}{t}\right) \lambda^2+t^2 \lambda-t\equiv f_2(\lambda),   \\
  	& \operatorname{det}\left(\lambda I-A^{0,1,0}_{-1,-3,0}\right)=\lambda^3+\left(2 t+\frac{1}{t}\right) \lambda^2+\left(t^2-1\right) \lambda-t\equiv f_3(\lambda),    \\
  	& \operatorname{det}(\lambda I-A_{-2,-2,0}^{1,0,0})=\lambda^3+\left(2 t+\frac{1}{t}\right) \lambda^2+\left(t^2-1\right) \lambda-t-\frac{1}{t}\equiv f_4(\lambda), \\
  	& \operatorname{det}\left(\lambda I-A_{-2,-1,-1}^{1,0,0}\right)=\lambda^3+\left(2 t+\frac{1}{t}\right) \lambda^2+\left(t^2-1\right) \lambda-2 t-\frac{1}{t} \equiv f_5(\lambda),\\
  	& \operatorname{det}\left(\lambda I-A_{-1,-3,-1}^{1,1,1}\right)=\lambda^3+\left(t+\frac{3}{t}\right) \lambda^2+\frac{3}{t^2} \lambda-\frac{1}{t}+\frac{1}{t^3}  \equiv f_6(\lambda),\\
  	& \operatorname{det}\left(\lambda I-A_{-2,-2,-1}^{1,1,1}\right)=\lambda^3+\left(t+\frac{3}{t}\right) \lambda^2+\frac{3}{t^2} \lambda-t-\frac{1}{t}+\frac{1}{t^3} \equiv f_7(\lambda).
   \end{align*}
  Notice that
  $$f_1(\lambda)>f_2(\lambda)>f_3(\lambda)>f_4(\lambda)>f_5(\lambda)\quad \text{and}\quad f_6(\lambda)>f_7(\lambda).$$
  By Lemma \ref{lemm13}, the minimum spectral radius of $G$ is the maximum real root  of $f_1(\lambda)$ or the maximum real root of $ f_6(\lambda)$.
  Since $f_1(0)<0$, $f_6(0)<0$, $f_1(1/t^3)>0$, $f_6(1/t^3)<0$, and $f_1(\lambda)$, $f_6(\lambda)$ are both increasing for $\lambda>0$, we conclude that
  $$G\cong G(1,0,0;m-1, m-2, m-1).$$
   \textbf{Case 2. }$n=6m+1$. By Claim 4 we have
  $$G\in\left\{\begin{array}{ll}G(1,0,1;m-1,m-2,m-1),&G(1,0,1;m-2,m-1,m-1),\\G(1,1,0;m-1,m-2,m-1),&G(1,1,0;m-2,m-2,m),\\G(0,0,0;m-1,m-1,m-1),&G(0,0,0;m-1,m-2,m)\\\end{array}
  \right\}.$$
  By  equation (\ref{equ2}),  we have
  \begin{align*}
  		& \operatorname{det}\left(\lambda I-A_{-1,-2,-1}^{1,0,1}\right)=\lambda^3+\left(t+\frac{2}{t}\right) \lambda^2+\frac{1}{t^2} \lambda-\frac{1}{t}\equiv g_1(\lambda),	\\
  		& \operatorname{det}\left(\lambda I-A_{-2,-1,-1}^{1,0,1}\right)=\lambda^3+\left(t+\frac{2}{t}\right) \lambda^2+\left(\frac{1}{t^2}-1\right) \lambda-t-\frac{2}{t},\\
  		& \operatorname{det}\left(\lambda I-A_{-1,-2,-1}^{1, 1,0}\right)=\lambda^3+\left(t+\frac{2}{t}\right) \lambda^2+\frac{1}{t^2} \lambda-t,\\
  		& \operatorname{det}\left(\lambda I-A_{-2,-2,0}^{1,1,0}\right)=\lambda^3+\left(t+\frac{2}{t}\right) \lambda^2+(\frac{1}{t^2}-1) \lambda-\frac{1}{t}-t,\\
  		& \operatorname{det}\left(\lambda I-A_{-1,-2,0}^{0,0,0}\right)=\lambda^3+2 t \lambda^2+\left(t^2-2\right) \lambda-t\equiv g_2(\lambda), \\
  		& \operatorname{det}\left(\lambda I-A_{-1,-1,-1}^{0,0,0}\right)=\lambda^3+2 t \lambda^2+\left(t^2-2\right) \lambda-2 t.
  		\end{align*}
Notice that $g_1(0)<0$, $g_1\left(1/t\right)=4/t^3>0$,  $g_2(0)<0$,
  		$g_2\left(1/t\right)=1/t^3>0$, $g_1(\lambda)$, $g_2(\lambda)$ are both increasing for $\lambda>0$. Moreover, since $g_1(\lambda)-g_2(\lambda)$ is monotonically decreasing in $\lambda\in(0,1/t)$, we have  $g_1(\lambda)-g_2(\lambda)>0$ for $\lambda\in(0,1/t)$.   Similarly as in Case 1, by Lemma \ref{lemm13} and Lemma \ref{lemma14} we can conclude that
  	$$G\cong G(1,0,1;m-1,m-2,m-1).$$
  	 \textbf{Case 3. }$n=6m+2$. By Claim 4 we have
   $$G\in\left\{\begin{array}{ll}G(1,0,0;m-1, m-2, m),&G(0,1,0;m-1, m-2, m),\\G(1,0,0;m-1, m-1, m-1),&G(1,1,1;m-1, m-2, m-1),\\G(0,1,0;m, m-3, m),&G(1,0,0;m-2, m-1, m)\\\end{array}
   \right\}.$$
    By  equation (\ref{equ2}),  we have
  	 \begin{align*}
  	 	& \operatorname{det}\left(\lambda I-A_{-1,-2,0}^{1,0,0}\right)=\lambda^3+\left(t+\frac{1}{t}\right) \lambda^2-\lambda-\frac{1}{t} \equiv g_3(\lambda),\\
  	 	& \operatorname{det}\left(\lambda I-A_{-1,-2,0}^{0,1,0}\right)=\lambda^3+\left(t+\frac{1}{t}\right) \lambda^2-\lambda-t, \\
  	 	& \operatorname{det}\left(\lambda I-A_{-1,-1,-1}^{1,0,0}\right)=\lambda^3+\left(t+\frac{1}{t}\right) \lambda^2-\lambda-t-\frac{1}{t}, \\
  	 	 &\operatorname{det}\left(\lambda I-A_{-1,-2,-1}^{1,1,1}\right)=\lambda^3+\frac{3}{t} \lambda^2+\left(\frac{3}{t^2}-2\right) \lambda-\frac{2}{t}+\frac{1}{t^3},\\
  	 	& \operatorname{det}\left(\lambda I-A_{0,-3,0}^{0,1,0}\right)=\lambda^3+\left(t+\frac{1}{t}\right) \lambda^2-2 \lambda \equiv g_4(\lambda),\\
  	 	& \operatorname{det}\left(\lambda I-A_{-2,-1,0}^{1,0,0}\right)=\lambda^3+\left(t+\frac{1}{t}\right) \lambda^2-2 \lambda-t-\frac{1}{t}.
  \end{align*}
  Notice that $g_3(1/t)=g_4(1/t)<0$, $g_3(\lambda)$, $g_4(\lambda)$ are both increasing for $\lambda>1/t$ and $g_3(\lambda)-g_4(\lambda)>0$ for $\lambda>1/t$.	Similarly as in Case 1, by Lemma \ref{lemm13} we conclude that  $$G\cong G(1,0,0;m-1,m-2,m).$$
  	\textbf{Case 4.} $n=6m+3$. By Claim 4 we have
  $$G\in\left\{\begin{array}{ll}G(0,0,0;m,m-2,m),&G(0,0,0;m-1,m-1,m),\\G(1,1,0;m-1,m-2,m),&G(1,0,1;m-1,m-1,m-1)\\\end{array}
  \right\}.$$
   By  equation (\ref{equ2}),  we have
  	 \begin{align*}
  			& \operatorname{det}\left(\lambda I-A_{0,-2,0}^{0,0,0}\right)=\lambda^3+t \lambda^2-2 \lambda \equiv g_5(\lambda),\\
  	 		& \operatorname{det}\left(\lambda I-A_{-1,-1,0}^{0,0, 0}\right)=\lambda^3+t \lambda^2-2 \lambda-t, \\
  	 		& \operatorname{det}\left(\lambda I-A_{-1,-2,0}^{1,1,0}\right)=\lambda^3+\frac{2}{t} \lambda^2+\left(\frac{1}{t^2}-2\right) \lambda-\frac{1}{t}\equiv g_6(\lambda), \\
  	 		& \operatorname{det}\left(\lambda I-A_{-1,-1,-1}^{1,0,1}\right)=\lambda^3+\frac{2}{t} \lambda^2+\left(\frac{1}{t^2}-2\right) \lambda-\frac{2}{t}.  \end{align*}
  	 	Similarly as in Case 1, by Lemma \ref{lemm13} since $g_5(1)>0$, $g_6(1)<0$ and $g_5(\lambda)$, $g_6(\lambda)$ are both increasing for $\lambda>1$, we conclude that  $$G\cong G(0,0,0;m,m-2,m).$$
  	 \textbf{Case 5.} $n=6m+4$. By Claim 4 we have
  	 $$G\in\left\{\begin{array}{ll}G(0,1,0;m, m-2, m),&G(1,0,0;m-1, m-1, m),\\G(1,1,1;m, m-2, m-1),&G(1,1,1;m-1, m-1, m-1)\\\end{array}\right\}.$$
  	 By  equation (\ref{equ2}),  we have
  	 \begin{align*}
  	 	& \operatorname{det}\left(\lambda I-A^{0,1,0}_{0,-2,0}\right)=\lambda^3+\frac{1}{t} \lambda^2-2 \lambda \equiv g_7(\lambda),\\
  	 	& \operatorname{det}\left(\lambda I-A_{-1,-1,0}^{1, 0,0}\right)=\lambda^3+\frac{1}{t} \lambda^2-2 \lambda-\frac{1}{t}, \\
  	 	& \operatorname{det}\left(\lambda I-A_{0,-2,-1}^{1, 1,1}\right)=\lambda^3+\left(\frac{3}{t}-t\right) \lambda^2+\left(\frac{3}{t^2}-4\right) \lambda+t-\frac{3}{t}+\frac{1}{t^3}\equiv g_8(\lambda),  \\
  	 	& \operatorname{det}\left(\lambda I-A_{-1,-1,-1}^{1,1,1}\right)=\lambda^3+\left(\frac{3}{t}-t\right) \lambda^2+\left(\frac{3}{t^2}-4\right) \lambda-\frac{3}{t}+\frac{1}{t^3}.  \end{align*}
  	 Similarly as in Case 1, by Lemma \ref{lemm13} since  $ g_7(1)<0$, $g_8(1)<0$ and $g_7(\lambda)-g_8(\lambda)>0$ for $\lambda>1$, we conclude that $$G\cong G(0,1,0;m,m-2,m).$$
  	 \textbf{Case 6.}  $n=6m+5$. By Claim 4 we have
  	 $$G\in\left\{\begin{array}{ll}G(0,0,0;m,m-1,m),&G(1,1,0;m,m-2,m),\\G(1,0,1;m-1,m,m-1),&G(1,0,1;m,m-1,m-1),\\G(1,1,0;m-1,m-1,m),&G(1,1,0;m-1,m-2,m+1)\\\end{array}\right\}.$$  By  equation (\ref{equ2}),  we have
  	 	\begin{align*}
  	 			&\operatorname{det}\left(\lambda I-A_{0,-1,0}^{0,0, 0}\right)=\lambda^3-2\lambda\equiv g_{9}(\lambda),\\
  	 		& \operatorname{det}\left(\lambda I-A_{0,-2,0}^{1,1,0}\right)=\lambda^3+\left(\frac{2}{t}-t\right) \lambda^2+\left(\frac{1}{t^2}-3\right) \lambda+t-\frac{1}{t}\equiv g_{10}(\lambda), \\
  	 	& \operatorname{det}\left(\lambda I-A_{-1,0,-1}^{1,0,1}\right)=\lambda^3+\left(\frac{2}{t}-t\right) \lambda^2+\left(\frac{1}{t^2}-4\right) \lambda-\frac{3}{t}, \\
  	 	& \operatorname{det}\left(\lambda I-A_{	0,-1,-1}^{1,0,1}\right)=\lambda^3+\left(\frac{2}{t}-t\right) \lambda^2+\left(\frac{1}{t^2}-3\right) \lambda+t-\frac{2}{t} ,\\
  	 	& \operatorname{det}\left(\lambda I-A_{-1,-1,0}^{1, 1,0}\right)=\lambda^3+\left(\frac{2}{t}-t\right) \lambda^2+\left(\frac{1}{t^2}-3\right) \lambda-\frac{1}{t} ,\\
  	 	& \operatorname{det}\left(\lambda I-A_{-1,-2,1}^{1,1,0}\right)=\lambda^3+\left(\frac{2}{t}-t\right) \lambda^2+\left(\frac{1}{t^2}-4\right) \lambda+t-\frac{2}{t}.
  	 \end{align*}
 Since $g_{10}(t)=0$, we have $ \lambda_1\left(A_{0,-1,0}^{0, 0,0}\right)<\lambda_1\left(A_{0,-2,0}^{1,1,0}\right)$. Similarly as in Case 1, by Lemma \ref{lemm13} we conclude that $$G\cong G(0,0,0;m,m-1,m).$$
 \par Combine all the above cases we completes the proof of Theorem \ref{theo1}. \qed
 \par
\begin{rema}
	Suppose $G$ attains the minimum spectral radius in $\mathcal{G}_{n,n-3}$. If $n=5$, $G\cong C_4\vee K_1$ (See graph 1.13 in the page 273 of \cite{Cvetkovi1995}). If $n=6$, $G\cong K_{3,3}-e$, where $e$ is an edge of $K_{3,3}$ (See graph 74 in \cite{CP}). If $n\in\{7,8\}$, $n-3=\left\lfloor\frac{2n}{3}\right\rfloor$ and $G\cong C_n$. If $n\in\{9,10,11\}$, $n-3=\left\lceil\frac{2n}{3}\right\rceil$ and $G\cong  P_n$. For the case   $12\le n\le39$, , a more detailed structural analysis could be added to prove that the extremal graphs also satisfy Theorem \ref{theo1}. This analysis was omitted to maintain the readability of the paper.
\end{rema}
\section*{Acknowledgement}
This work was supported by the National Natural Science Foundation of China (No. 12171323).

\end{document}